\documentclass[12pt,reqno]{amsart}

\advance \topmargin by -\headheight
\advance \topmargin by -\headsep
     
\evensidemargin \oddsidemargin
\usepackage{amsmath}
\usepackage{amsfonts}
\usepackage{stmaryrd}
\usepackage{amssymb}
\usepackage{amsthm}
\usepackage{csquotes}

\usepackage{mathrsfs}
\usepackage{dsfont}
\usepackage{bm}

\numberwithin{equation}{section}

\newtheorem{theorem}{Theorem}[section]

\newtheorem{lemma}[theorem]{Lemma}
\newtheorem{Proposition}[theorem]{Proposition}
\newtheorem{Conjecture}[theorem]{Conjecture}
\newtheorem{Corollary}[theorem]{Corollary}

\title[Sums of Linear Transformations]{Sums of Linear Transformations in higher dimensions}
\author[Akshat Mudgal]{Akshat Mudgal}
\date{} 
\address{ School of Mathematics, University of Bristol, University Walk, Clifton, Bristol BS8 1TW, United Kingdom}
\email{am16393@bristol.ac.uk}

\begin{document}

\maketitle
\begin{abstract}
In this paper, we prove the following two results. Let $d$ be a natural number and $q,s$ be co-prime integers such that $1 < qs$. Then there exists a constant $\delta > 0$ depending only on $q,s$ and $d$ such that for any finite subset $A$ of $\mathbb{R}^d$ that is not contained in a translate of a hyperplane, we have
\[  |q\cdot A + s\cdot A| \geq  (|q| +|s|+ 2d-2)|A|        - O_{q,s,d}(|A|^{1-\delta}) . \] The main term in this bound is sharp and improves upon an earlier result of Balog and Shakan. Secondly, let $\mathscr{L} \in \textrm{GL}_{2}( \mathbb{R})$ be a linear transformation such that $\mathscr{L}$ does not have any invariant one-dimensional subspace of $\mathbb{R}^2$. Then for all finite subsets $A$ of $\mathbb{R}^2$, we have
\[  |A + \mathscr{L}(A)| \geq  4|A| - O(|A|^{1-\delta}), \]
for some absolute constant $\delta > 0$. The main term in this result is sharp as well. 
\end{abstract}

\section{Introduction}
Let $A, B$ be finite subsets of $\mathbb{R}^d$, for some $d \in \mathbb{N}$. We define 
\[ A + B = \{ a + b \ | \ a \in A, b \in B \}. \]
Furthermore, for all real numbers $q$, and $a = (a_1, \dots, a_d) \in \mathbb{R}^d$, we define
\[ q\cdot a = (qa_1, \dots,  qa_d), \]
and for all $A \subseteq \mathbb{R}^d$, 
\[ q\cdot A = \{ q\cdot a \ | \ a \in A\}. \]
We define \emph{dimension} of a set $A \subseteq \mathbb{R}^d$ to be the dimension of the affine subspace spanned by $A$. Our first result is on sums of dilates.

\begin{theorem} \label{main}
Let $d$ be a natural number and $q,s$ be co-prime integers such that $1 < qs$. Further, let $A$ be a finite $d$-dimensional subset of $\mathbb{R}^d$. Then there exists a constant $\delta > 0$ depending only on $q,s$ and $d$ such that
\[  |q\cdot A + s\cdot A| \geq  (|q| +|s|+ 2d-2)|A|        - O_{q,s,d}(|A|^{1-\delta}) . \]
\end{theorem}

The constant $|q| + |s| + 2d-2$ in Theorem $\ref{main}$ is sharp as witnessed by the following example. Let $e_1, e_2, \dots, e_d$ be the standard basis for $\mathbb{Z}^d$. For each $N \in \mathbb{N}$, define 
\[A_N = \{e_1,e_2, \dots, e_d\} \cup \{ 2e_1, \dots, Ne_1 \}. \]
An easy computation shows that
\[ |q\cdot A_N + s\cdot A_N| \leq (|q| + |s| + 2d - 2)|A_N| - O_{q,s,d}(1). \]

In the case $d=1$, a generalization of Theorem $\ref{main}$ to sums of several dilates with a better error term was proved by Shakan \cite{Sh2016}. Furthermore, when $d \geq 2$, previously best known lower bounds for $|A+q\cdot A|$ were by Balog and Shakan \cite{BalSh15}. When $d \geq 4$ and $q \geq 2$, they showed that 
\[ |A+q\cdot A| \geq (q + d + 1)|A| - O_{q,d}(1). \]
Furthermore, in the same paper, they showed that when $d \in \{ 2,3 \}$ and $q\geq 2$, 
\[ |A+q\cdot A| \geq (q + 2d - 1)|A| - O_{q,d}(1), \]
which they conjectured to be true for all $d \in \mathbb{N}$. 

\begin{Conjecture} \label{BalShConj}
Let $d, q$ be natural numbers such that $q > 1$ and let $A \subseteq \mathbb{Z}^d$ be a finite $d$-dimensional set. Then
\[|A + q \cdot A| \geq (q + 2d - 1)|A| - O_{q,d}(1). \] 
\end{Conjecture}
We observe that Theorem $\ref{main}$ implies Conjecture $\ref{BalShConj}$ with a slightly worse error term. 
\par

Our second result is about sums of linear transformations in $\mathbb{R}^2$. Firstly, given $\mathscr{L} \in \textrm{GL}_{d}( \mathbb{R})$  and $A \subseteq \mathbb{R}^d$, we define
\[\mathscr{L}(A) = \{\mathscr{L}(a) \ | \ a \in A \}.  \] 
We give lower bounds for $|A + \mathscr{L}(A)|$ where $A \subseteq \mathbb{R}^2$ and $\mathscr{L} \in \textrm{GL}_{2}( \mathbb{R})$ such that $\mathscr{L}$ does not have any invariant one-dimensional subspace of $\mathbb{R}^2$. 

\begin{theorem} \label{rtn}
Let $A$ be a finite subset of $\mathbb{R}^2$. Furthermore, let $\mathscr{L} \in \textrm{GL}_{2}( \mathbb{R})$ be a linear transformation such that $\mathscr{L}$ has no real eigenvalues. Then there exists an absolute constant $\delta > 0$, such that
\[  |A + \mathscr{L}(A)| \geq  4|A| - O(|A|^{1-\delta}) . \]
\end{theorem}

In particular, we can choose $\mathscr{L} = \mathscr{L}_{\theta}$  for some $\theta \in (0, 2\pi) \setminus \{ \pi \}$, where $\mathscr{L}_{\theta}$ rotates vectors in $\mathbb{R}^2$ counterclockwise by angle $\theta$. As  $\theta \in (0, 2\pi) \setminus \{ \pi \}$, we see that $\mathscr{L}_{\theta}$ has no real eigenvalues. 

\begin{Corollary} \label{rotation}
Let $A$ be a finite subset of $\mathbb{R}^2$ and $\theta \in (0, 2\pi) \setminus \{ \pi \}$. Then we have
\[ |A + \mathscr{L}_{\theta}(A)| \geq  4|A| - O(|A|^{1-\delta}),\] 
for some absolute constant $\delta > 0$. 
\end{Corollary}

The main term in our lower bound is sharp as witnessed by the following example. Let 
\[B_N = \{ (a, b) \ | \ 0 \leq a, b \leq N-1 \} \cap \mathbb{Z}^2, \]
and $\theta = \pi/2$. In this case, we see that
\[ \mathscr{L}_{\theta}(B_N) = B_N - \{(N-1,0) \}, \]
and thus
\[ |B_N + \mathscr{L}_{\theta}(B_N)|  =  |B_N + B_N|  \leq  |B_{2N}|  = 4|B_N|. \]
Note that if $\theta \in \{ 0, \pi \}$, one can take $A$ to be a $1$-dimensional arithmetic progression and show that 
\[ |A + \mathscr{L}_{\theta}(A)| = 2|A| - O(1),  \]
which is best possible, as for any two finite, non-empty subsets $A, B$ of $\mathbb{R}^2$, one has
\[ |A+B| \geq |A| + |B| - 1.\]
Further, if one restricts $A$ to be $2$-dimensional and $\theta \in [0, 2\pi)$, the best lower bound that can be shown is 
\begin{equation} \label{glb} 
|A+ \mathscr{L}_{\theta}(A)| \geq 3|A| - 3,
\end{equation}
which follows from a result of Ruzsa \cite[Corollary 1.1]{Ru1994}. It is sharp for $\theta = 0$ and $\theta = \pi$ as the set
\[ C_N = \{0, e_2 \} +  \{ t \cdot e_1 \ | \ t \in \{1,2,\dots,N-1 \} \}, \]
 demonstrates. Hence when $\theta \in \{0, \pi\}$ and $A$ is $2$-dimensional, the best lower bound that we can get is $\eqref{glb}$. Corollary \ref{rotation} implies that for all other values of $\theta$, one can get a stronger lower bound for $|A + \mathscr{L_{\theta}}(A)|$.
\par
We will deduce Theorem $\ref{main}$ and Theorem $\ref{rtn}$ from a structure theorem for sets with few sums of linear transformations. 

\begin{theorem} \label{str}
Let $c$ be a positive real number and let $d$ be a natural number. Further, let $A$ be a finite subset of $\mathbb{R}^d$ and $\mathscr{L} \in \textrm{GL}_d(\mathbb{R})$ be an invertible linear transformation. If 
\[ |A+\mathscr{L}(A)| \leq c|A|, \]
then there exist parallel lines $l_1, l_2, \dots, l_r$ in $\mathbb{R}^d$, and constants $0 < \sigma \leq 1/2$ and $C_1>0$ depending only on $c$ such that 
\[ |A \cap l_1| \geq \dots \geq  |A \cap l_r| \geq  |A \cap l_1|^{1/2} \geq C_1^{-1} |A|^{\sigma}, \]
and 
\[ |A\setminus (l_1 \cup l_2 \cup \dots \cup l_r)| < C_1 c^6 |A|^{1-\sigma}. \]
\end{theorem}

We note that the problem of looking at sums of dilates in vector spaces is a generalisation of estimating lower bounds for sums of dilates of subsets of integers.  Originally, Konyagin and \L aba \cite{KL2006} worked on sets of the form $A + \lambda \cdot A$ for $A \subseteq \mathbb{R}$ and transcendental $\lambda$. Subsequently, Nathanson \cite{Na2008} gave lower bounds for $|A + \lambda \cdot A|$ when $A \subseteq \mathbb{Z}$ and $\lambda \in \mathbb{N} \setminus \{1\}$. Different variants of this problem were tackled by many authors (see \cite{BalSh13},  \cite{CHS2009}, \cite{CSV2010}, \cite{DCS2014}, \cite{HR2011} and \cite{Lj2013}) and in particular, the general case of estimating $|\lambda_1 \cdot A + \dots + \lambda_k \cdot A|$ for co-prime integers $\lambda_1, \dots, \lambda_k$ was first treated by Bukh \cite{Bu2008}. Bukh gave a lower bound for size of such sets and the main term in Bukh's bound was sharp. The final improvement for Bukh's error term was given by Shakan \cite{Sh2016}. As previously mentioned, this result was generalised to $d$-dimensional subsets of $\mathbb{Z}^d$ by Balog and Shakan in \cite{BalSh15}. We refer the reader to \cite{BalSh13}, \cite{Bu2008} and \cite{Sh2016} for a more detailed introduction to this problem.
\par

We remark that there are multiple variants of this problem that are currently unsolved and are of independent interest. In \cite[Corollary 3.7]{KL2006}, Konyagin and \L aba proved that for any transcendental real number $\lambda$ and finite set $A \subseteq \mathbb{R}$ such that $|A|>1$, one has
\[ |A + \lambda \cdot A | = \Omega(|A| \log|A| / \log{\log|A|}).  \]
They further showed that there exist arbitrarily large sets $A$ with
\[ |A + \lambda \cdot A |  = \exp(O( \log^{1/2} |A|)) |A|. \] 
There were subsequent improvements to Konyagin and \L aba's result by Sanders \cite{Sa2008}, \cite{Sa2012} and Schoen \cite{Sc2011}. In particular, Sanders \cite[Theorem 11.8]{Sa2012} showed that one can improve Konyagin and \L aba's lower bound to
\[ |A +  \lambda \cdot A | = \exp(\Omega( \log^{\Omega(1)} 2|A|)) |A|. \]
It would be interesting to find the exact shape of a sharp lower bound for $|A + \lambda \cdot A |$ when $\lambda$ is a transcendental real number. 
\par

Similarly, one might be interested in estimates for $| A + \lambda \cdot A  |$ when $\lambda$ is an algebraic number and $A \subseteq \mathbb{Z} [\lambda]$. As Shakan remarks in \cite[Question 1.2]{Sh2016}, this is closely related to a conjecture of Bukh that asks for lower bounds for $| \mathscr{L}_1(A) + \dots + \mathscr{L}_k(A)|$ where $A \subseteq \mathbb{Z}^d$ and $\mathscr{L}_1, \dots, \mathscr{L}_k$ are linear transformations from $\mathbb{Z}^d$ to $\mathbb{Z}^d$. 

\begin{Conjecture} \label{BkhSh}
Let $\mathscr{L}_1, \dots, \mathscr{L}_k$ be linear transformations from $\mathbb{Z}^d$ to $\mathbb{Z}^d$ that do not share a non-trivial invariant subspace and satisfy 
\[ \mathscr{L}_1 (\mathbb{Z}^d) + \dots +  \mathscr{L}_k (\mathbb{Z}^d) = \mathbb{Z}^d.\]  
Then for any $A \subseteq \mathbb{Z}^d$, we have
\[  |\mathscr{L}_1(A) + \dots + \mathscr{L}_k(A)| \geq \big(  |\det(\mathscr{L}_1)|^{1/d} + \dots + |\det(\mathscr{L}_k)|^{1/d}  \big)^{d} |A| - o(|A|). \] 
\end{Conjecture}

We observe that one can conjecture a similar result for linear transformations from $\mathbb{R}^d$ to $\mathbb{R}^d$. In \S5, we present a structure theorem, that is, Theorem $\ref{costr}$, which makes partial progress towards an analogue of Conjecture $\ref{BkhSh}$ in $\mathbb{R}^2$. Furthermore, Theorem $\ref{costr}$ implies Theorem $\ref{rtn}$ in a straightforward manner, which in itself, shows that Conjecture $\ref{BkhSh}$ is true when $d,k=2$ and $\mathscr{L}_1, \mathscr{L}_2$ are linear transformations from $\mathbb{R}^2$ to $\mathbb{R}^2$ with $\mathscr{L}_1$ as the identity matrix and $|\det(\mathscr{L}_2)| = 1$. 
\par

Lastly, this problem can also be considered in the finite field setting, that is, given a prime $p$ and $A \subseteq \mathbb{F}_p$, we look at $A + q \cdot A$ where $q \in \mathbb{F}_p$. When $q=1$, the question is answered by the Cauchy--Davenport theorem. But for general values of $q$, the question remains open, with partial results in \cite{Pl2011} and \cite{Po2013}. 

 \par
We now outline the structure of our paper. We dedicate \S2 to present some preliminary results that we will use in our paper. In \S3  we will prove Theorem $\ref{str}$. We use \S4 to combine Theorem $\ref{str}$ with some counting arguments from Combinatorial Geometry to show Theorem $\ref{main}$. Lastly, in \S5, we prove Theorem $\ref{costr}$ and Theorem $\ref{rtn}$. \par

\section{Preliminaries}

In our proof of Theorem $\ref{str}$, we will use two standard inequalities to move from sum of dilates to sumsets. The first of these two inequalities was originally shown by Ruzsa \cite{Ru1996}. We mention these results as stated in \cite[Lemma 2.6]{TV2006} and \cite[Corollary 2.12]{TV2006}.

\begin{lemma} \label{rusza}
Suppose that $U,V,W$ are three finite sets in some abelian group $G$. Then 
\begin{equation} \label{ru1} |U||V-W| \leq |U-V||U-W|  \end{equation}
and
\begin{equation} \label{ru2} |U+V| \leq \frac{|U-V|^3}{|U||V|}.  \end{equation}
\end{lemma}

Another important ingredient for the proof of Theorem $\ref{str}$ will be the following generalisation of Freiman's theorem on sets with small doubling to arbitrary abelian groups by Green and Ruzsa \cite{GR2007}. In order to state the result, we have to give some additional definitions. Given an abelian group $G$, we define a \emph{proper progression} $P$ of \emph{arithmetic dimension} $s$ and size $L$ as
\[  P = \{   v_0 + u_1 v_1 + \dots + u_s v_s \ | \ 0 \leq u_i < L_i \ (1 \leq i \leq s)  \},  \]
where $L_1 L_2 \dots L_s = L$ and $v_0, v_1, \dots, v_s$ are elements of $G$ such that all the sums in the progression are distinct. We further define a \emph{coset progression} to be a set of the form $P+H$ where $P$ is a proper progression and $H$ is a subgroup of $G$. It is important to not confuse the arithmetic dimension of a progression $P$ as defined above and the dimension of a subset $A$ of $\mathbb{R}^d$ as defined earlier to be the dimension of the affine subspace spanned by $A$. We now state Green and Ruzsa's result  \cite[Theorem 1.1]{GR2007}.

\begin{lemma} \label{GrRu}
Let $A$ be a subset of an abelian group $G$ such that $|A+A| \leq K|A|$. Then $A$ is contained in a coset progression of arithmetic dimension $s \leq CK^4 \log{(K+2)}$ and size $L = |P+H| \leq \ e^{CK^4 {\log}^2 {(K+2)}} |A|$, for some constant $C > 0$.  
\end{lemma}

As a remark, we note that Lemma $\ref{GrRu}$ has been quantitatively improved by many authors (for instance, see \cite{San2012}, \cite{Sc2011}). In particular, much work has been done on improving the dependence of $s$ and $L$ on $K$. At the same time, we observe that Theorem $\ref{str}$ refers to the existence of constants $0 < \sigma \leq 1/2$, and $C_1>0$  such that the theorem holds and does not deal with the quantitative dependence of $\sigma$ and $C_1$ on $c$. Thus, for our purposes, it suffices to use Lemma $\ref{GrRu}$ as stated.  
\par

Note that if the group $G$ is torsion free, then the finite subgroup $H$ must be trivial for finite $A$. Thus if $A$ is a subset of $\mathbb{Z}^d$ or $\mathbb{R}^d$ and $A$ has small doubling, then $A$ must lie in a proper progression $P$ of bounded arithmetic dimension and size proportional to size of $A$. 
\par

In our proof of Theorem $\ref{main}$, we will frequently use a straightforward consequence of a result of Shakan \cite[Theorem 1.1]{Sh2016}.
\begin{lemma} \label{BalSh}
Given distinct co-prime integers $q,s$ there exists a constant $C_{q,s}$ such that for every finite subset $A$ of $\mathbb{Z}$, one has
\[ |q\cdot A + s\cdot A| \geq (|q|+|s|)|A| - C_{q,s}.    \]
\end{lemma}

In fact, Balog and Shakan give an explicit upper bound for the additive constant $C_{q,s}$. In \cite{Sh2016}, Shakan remarks that results like Lemma $\ref{BalSh}$ can be extended to $A \subseteq \mathbb{R}$ by using a result from \cite[Lemma 5.25]{TV2006}. For completeness, we record the same below.

\begin{lemma} \label{BalSh1}
Given distinct co-prime integers $q,s$ there exists a constant $C_{q,s}$ such that for every finite subset $A$ of $\mathbb{R}$, one has
\[ |q\cdot A + s\cdot A| \geq (|q|+|s|)|A| - C_{q,s}.    \]
\end{lemma}

Note that as sums of dilates are preserved under invertible linear transformations, we can deduce that given a finite $1$-dimensional set $A \subset \mathbb{R}^d$ and distinct co-prime integers $q$ and $s$, there exists a constant  $C_{q,s}$ such that one has
\begin{equation} \label{bs13} |q\cdot A+s\cdot A| \geq (|q|+|s|)|A| - C_{q,s}. \end{equation}
Another result which we will use is a result on $d$-dimensional sumsets in $\mathbb{R}^d$ by Ruzsa \cite[Corollary 1.1]{Ru1994}. 

\begin{lemma} \label{ruzsa}  
Let $A, B$ be finite, non-empty subsets of $\mathbb{R}^d$ such that $|A| \geq |B|$ and $dim(A+B) = d$. Then we have
\[ |A+B| \geq |A| + d|B| - d(d+1)/2.\]
\end{lemma}

In some instances, we will also use a more general lower bound for sumsets of arbitrary finite sets in $\mathbb{R}^d$. Thus, given any finite, non-empty sets $A,B \subseteq \mathbb{R}^d$, we have
\begin{equation} \label{usls}
 |A+B| \geq |A| + |B| - 1. \end{equation}

Lastly, in \S5, we will use a result of Grynkiewicz and Serra \cite[Theorem 1.3]{GS2010}. 

\begin{lemma} \label{gs}
Let $A, B \subseteq \mathbb{R}^2$ be finite, non-empty subsets, let $l = \mathbb{R} x_1$ be a line, let $r_1$ be the number of lines parallel to $l$ which intersect $A$, and let $r_2$ be the number of lines parallel to $l$ that intersect $B$. Then
\begin{equation}  \label{gs1}
 |A+B| \geq \bigg( \frac{|A|}{r_1} + \frac{|B|}{r_2} - 1 \bigg)  (r_1 + r_2 - 1). 
\end{equation}
\end{lemma}

\section{The structure theorem}

In this section, we will prove Theorem $\ref{str}$. We begin by moving from estimates on sums of dilates to bounds on sumsets. 

\begin{lemma} \label{sdss}
Let $c$ be a positive real number and let $d$ be a natural number. Further, let $A$ be a finite subset of $\mathbb{R}^d$ and $\mathscr{L} \in \textrm{GL}_d(\mathbb{R})$ be an invertible linear transformation. If 
\[ |A+\mathscr{L}(A)| \leq c|A|, \]
then
\begin{equation} \label{sum} |A+A|  \leq c^6 |A|. \end{equation}
\end{lemma}

\begin{proof}

Fixing $c>0$, let $A$ be a finite subset of $\mathbb{R}^d$ such that $|A+\mathscr{L}(A)| \leq c|A|$. We apply $\eqref{ru1}$ with $U = A$, $V = -\mathscr{L}(A)$ and $W = -\mathscr{L}(A)$. Thus, we have
\[ |A| |-\mathscr{L}(A) - (-\mathscr{L}(A))| \ \leq \ |A-(-\mathscr{L}(A))| |A - (-\mathscr{L}(A))|, \]
which gives us
\[ |A||\mathscr{L}(A-A)| \ \leq \ |A+\mathscr{L}(A)|^2 \ \leq \ c^2 |A|^2.  \]
As $\mathscr{L}$ is invertible, we have $|\mathscr{L}(A-A)| = |A-A|$. Thus we deduce that
\[  |A-A| \ \leq \ c^2|A|.  \]
Using $\eqref{ru2}$ with $U,V = A$, we get
\[  |A+A| \ \leq \ \frac{|A-A|^3}{|A|^2} \ \leq \ c^6 |A|. \qedhere \] 
\end{proof}

Our next objective is to deduce Theorem $\ref{str}$ from $\eqref{sum}$. 

\begin{lemma} \label{fri}
Let $A$ be a finite subset of $\mathbb{R}^d$ with $|A| = n$ where $n$ is large enough. If
\begin{equation} \label{ti} |A+A| \leq c^6n, \end{equation}
for some $c > 0$, then there exist parallel lines $l_1, l_2, \dots, l_r$ in $\mathbb{R}^d$, and constants $0 < \sigma \leq 1/2$ and $C_1 > 0$ depending only on $c$ such that 
\[ |A \cap l_1| \geq \dots \geq  |A \cap l_r| \geq  |A \cap l_1|^{1/2} \geq C_1^{-1} n^{\sigma}. \]
and 
\[ |A\setminus (l_1 \cup l_2 \cup \dots \cup l_r)| < C_1 c^6 n^{1-\sigma}. \] 

\end{lemma}

\begin{proof}

Let $A$ be a finite subset of $\mathbb{R}^d$ which satisfies $\eqref{ti}$. From the note following Lemma \ref{GrRu}, we deduce that $A$ is contained in a proper progression $P \subseteq \mathbb{R}^d$, of arithmetic dimension $s$ and size $C_1 n$, where $s$ and $C_1$ depend only on $c$. We write $P$ as 
\[  P = \{   v_0 + u_1 v_1 + \dots + u_s v_s \ | \ 0 \leq u_i < L_i \ (1 \leq i \leq s)  \},  \]
where $L_1 L_2 \dots L_s = L$ and $v_0, v_1, \dots, v_s$ are elements of $G$ such that all the sums in the progression are distinct.
\par

Without loss of generality, we suppose $L_1 = \sup \{L_1, \dots, L_s\}$. Note that as $P$ contains $A$, we must have $L_1 L_2 \dots L_s = L \geq n$, which further implies that $L_1 \geq  n^{1/s}$. We define the arithmetic progression $Q$ as
\[ Q = \{ u_1 v_1 \ | \ 0 \leq u_1 < L_1  \}. \]
We note that our progression $P$ can be seen as a collection of $L/L_1$ translates of the arithmetic progression $Q$. Because $P$ is proper, all of these translates are disjoint and thus we have \[ \frac{L}{L_1} \ \leq \ \frac{C_1 n}{n^{1/s}} \ \leq \ C_1 n^{1-1/s}. \]
Lastly, as $A$ is covered by disjoint translates of $Q$, we define $Q'$ to be the translate of $Q$ containing the most elements of $A$. By the pigeonhole principle, we find that $Q'$ contains at least
\[ \frac{n}{C_1 n^{1-1/s}} \ = \ \frac{1}{C_1} n^{1/s} \]
elements of $A$.
\par

Until now, we have shown that if our set $A$ has small doubling, then a significant portion of its elements are contained in a $1$-dimensional progression. Our next goal is to show that unless almost all of $A$ is similarly structured, $|A+A|$ grows faster than just linearly in $A$. 
\par

 We let $l$ be the line in $\mathbb{R}^d$ that contains the arithmetic progression $Q$. We begin by covering $A$ with translates of $l$. Thus we have 
\[ A \subseteq l_1 \cup  l_2 \cup \dots \cup l_{k}, \]
where $l_1, l_2, \dots, l_k$ are parallel lines. We write $p_i = A \cap l_i$. Without loss of generality, we can assume that 
\begin{equation} \label{tada}
|p_1| \geq |p_2| \geq \dots \geq |p_{k}| \ \text{such that} \ |p_1| \ \geq \ \frac{1}{C_1} n^{1/s}. 
\end{equation}
Let $r \in \{1, \dots, k \}$ be a natural number such that 
\[ |p_1| \geq \dots \geq|p_{r}| \geq |p_1|^{1/2} > |p_{r+1}| \geq \dots \geq |p_k|. \]
We define $B = A \setminus (l_1 \cup \dots \cup \l_{r} )$. Note that 
\[ |B| = \sum_{j=r+1}^{k} |p_j|  < (k-r) |p_1|^{1/2}, \]
and thus
\[ k - r > |B||p_1|^{-1/2}. \]
Further, we see that
\begin{align*} |A+A| \ 
& \geq \ |p_1 + B| \ = \ \sum_{j=r+1}^{k} |p_1 + p_j| \\ 
& \geq \ (k-r) |p_1| \ > \ |p_1|^{1/2} |B|. \end{align*}
We combine this with $\eqref{ti}$ and $\eqref{tada}$ to show that 
\[ |B| \ < \ |p_1|^{-1/2} |A+A| \ \leq \ C_1^{1/2} c^6 n^{1-1/2s}. \]
We replace $1/2s$ with $\sigma$, and $C_1^{1/2}$ with $C_1$ to get Lemma $\ref{fri}$. 
\end{proof}

We note that upon combining Lemma $\ref{sdss}$ and Lemma $\ref{fri}$, we can deduce Theorem $\ref{str}$. 
\par

\section{Proof of Theorem $\ref{main}$}

We will take ideas from the proof of Freiman's lemma \cite[section 1.14]{Fr1973} as given in \cite[Lemma 5.13]{TV2006} and modify them to prove our own result.
\par

Let $|A| = n$ and $q,s$ be co-prime integers such that $1 < qs$. Let $n$ be large enough and $|q\cdot A+ s\cdot A| < 2(|q|+|s| + 2d-2)n$. We note that $|q\cdot A+ s\cdot A| = |A + (s/q) \cdot A|$ and thus, define $\mathscr{L}$ to be the scalar matrix $(s/q) I_d$, where $I_d$ is the $d \times d$ identity matrix. As $\mathscr{L}$ lies in $\textrm{GL}_d(\mathbb{R})$, we apply Lemma $\ref{sdss}$ to get
\[ |A+A| \leq (2(|q| + |s| + 2d-2))^6|A|.  \]
\par

Our next step is to apply Lemma $\ref{fri}$ with $c = 2(|q| + |s| + 2d-2)$. Thus, we can find parallel lines $l_1, l_2, \dots, l_{r_1}$ and constants $0 < \sigma \leq 1/2$, and $C_1>0$ depending only on $q, s$ and $d$ such that 
\[ |A \cap l_1| \geq \dots \geq  |A \cap l_{r_1}| \geq  |A \cap l_1|^{1/2} \geq C_1^{-1} n^{\sigma}. \]
and 
\[ |A\setminus (l_1 \cup l_2 \cup \dots \cup l_{r_1})| < C_1 c^6 n^{1-\sigma}. \]
Note that there is a natural upper bound for $r_1$ in terms of $n$ as
\[ n \ \geq \ \sum_{i=1}^{r_1} |A \cap l_i| \ \geq \ r_1 C_1^{-1} n^{\sigma}. \]
Thus $r_1 \leq C_1 n^{1-\sigma}. $
We write 
\[B = A\setminus (l_1 \cup l_2 \cup \dots \cup l_{r_1}).\]
Note that we can cover $B$ with translates of $l_1$, say, $l_{r_1+ 1}, \dots, l_r$. As for each $r_1 < i \leq r$, the line $l_i$ must contain at least one element of $B$, we have
\[ r-r_1 \leq |B|  < C_1 c^6 n^{1-\sigma}. \]
This, together with the estimates on $r_1$, implies that
\begin{equation} \label{rup11}  r < C_2 n^{1-\sigma}, \end{equation}
where $C_2$ is some positive constant that only depends only on $C_1$ and $c$. Thus we have proved that if $|q \cdot A + s \cdot A | < 2(|q| + |s| + 2d-2)|A|$, then $A$ can be written as 
\[ A = (l_1 \cup l_2 \cup \dots \cup l_r) \cap A, \] 
where $l_1, l_2, \dots, l_r$ are $r$ parallel lines in $\mathbb{R}^d$ and $r < C_2 |A|^{1-\sigma}$ for some constants $C_2 > 0$ and $0 < \sigma \leq 1/2$. 
\par

For ease of notation, we define $K_{q,s,d}$ as a positive constant depending only on $q, s$ and $d$ such that
\[ K_{q,s,d} := d(d+1)C_{q,s} + d(d+1) + C_{q,s}, \]
where $C_{q,s}$ is the constant referenced in Lemma $\ref{BalSh1}$.  

\begin{Proposition} \label{lines}
Let $d$ be a natural number and $q,s$ be co-prime integers such that $1 < qs$. Further, let $l_1, l_2, \dots, l_r$ be $r$ parallel lines in $\mathbb{R}^d$. Suppose $A$ is a finite $d$-dimensional subset of $\mathbb{R}^d$ such that 
\[ A \subseteq l_1 \cup l_2 \cup \dots \cup l_r . \]
Then we have 
\[  |q \cdot A + s \cdot A | \geq  (|q|+|s| + 2d-2)|A|  -  K r,\]
where $K = K_{q,s,d}$.
\end{Proposition}

We note that by combining the preceding discussion with $\eqref{rup11}$ and Proposition $\ref{lines}$, we can deduce Theorem $\ref{main}$ for $\delta = \sigma$.

\begin{proof}[Proof of Proposition $\ref{lines}$]

We will prove our proposition by double induction, first on $d$, that is, the dimension of $A$ and then on $r$, that is, the number of lines that make up $A$. For any choice of $d$ and $r$, we have $r \geq d$ as $A$ is a $d$-dimensional set. Let $P(d,r)$ be the statement of Proposition $\ref{lines}$ for $d$-dimensional sets $A$ which can be covered by $r$ parallel lines. Our base cases will be $P(1,r)$ for all $r \geq 1$ and $P(d,d)$ for all $d \geq 1$. In our inductive step, we will prove that if $P(d-1,r-1)$ and $P(d,r-1)$ are true, then $P(d,r)$ holds. We will thus conclude that $P(d,r)$ holds for all $r,d \in \mathbb{N}$ such that $r \geq d$.
\par

For ease of notation, let $p_i = A \cap l_i$ for all $1 \leq i \leq r$. We note that Lemma $\ref{BalSh1}$ implies $P(1,r)$ for all $r \geq 1$. Thus our remaining base case is $P(d,d)$ for all $d \geq 1$. This is easy to show since in this case, the sets $q\cdot p_i + s\cdot  p_j$ are disjoint for all $1 \leq i \leq j \leq r$.  
Hence for our $d$-dimensional set $A$, we have
\[  | q\cdot A + s\cdot A | = \sum_{\substack{i,j = 1 \\ i \neq j}}^{d}  |q\cdot p_i + s\cdot p_j| + \sum_{i=1}^{d} |q\cdot p_i+s\cdot p_i| . \] 
We use $\eqref{usls}$ to estimate $|q\cdot p_i + s\cdot p_j|$ and we use $\eqref{bs13}$ to estimate $|q\cdot p_i+s\cdot p_i|$. Thus, we get

\begin{align*}
  | q\cdot A + s\cdot A | & \geq \sum_{\substack{i,j = 1 \\ i \neq j}}^{d} (|p_i| + |p_j| -1) +  \sum_{i=1}^{d}( (|q|+|s|)|p_i| - C_{q,s})\\ 
& \geq (|q| + |s| + 2(d-1)) \sum_{i=1}^{d} |p_i| - C_{q,s} (d^2 + d) \\ 
& \geq  (|q|+|s|+2d-2)|A| - Kd. 
\end{align*}
\par

We now proceed with the inductive step, that is, for any $r,d \in \mathbb{N}$ such that $r > d$, we assume that $P(d-1,r-1)$ and $P(d,r-1)$ are true, and then prove $P(d,r)$. Thus let $A$ be a finite, $d$-dimensional subset of $\mathbb{R}^d$, such that $A \subseteq (l_1 \cup \dots \cup l_r)$, where $l_1, \dots, l_r$ are parallel. As all the $l_i$'s are parallel, let ${H}$ be the hyperplane orthogonal to $l_1$ and let $x_i$ denote the point of intersection of $H$ and $l_i$ for each $i$. We write $X = \{ x_1,\dots,x_{r}\}.$ Without loss of generality, we can assume that $x_{r}$ is an extreme point of $X$, that is, it is a vertex on the convex hull $C$ of $X$. We define $A' = A \setminus l_{r}$, $X' = X\setminus \{x_{r}\}$ and $C'$ to be the convex hull of $X'$. Note that dimension of $A'$ in $\mathbb{R}^d$ is at least $d-1$. Our proof divides into two cases now, depending on the dimension of $A'$.\par
We first consider the case when $A'$ is $d$-dimensional. This implies that $X'$ is $(d-1)$-dimensional, and since $x_{r}$ lies outside of $C'$, there exist distinct points $y_1,\dots,y_{d-1}$ in $X'$ such that for all $1\leq i \leq d-1$, the line segment joining $x_{r}$ and $y_i$ lies outside $C'$. In particular, for each $1\leq i \leq d-1$, the points
\[ \frac{s\cdot x_{r} + q\cdot y_i}{q+s} \ \text{and} \ \frac{q\cdot x_{r} + s\cdot y_i}{q+s}  \]
lie outside of $(q\cdot X' + s\cdot X')/(q+s)$. For each $y_i$, let the corresponding line in $A$ containing $y_i$ be $m_i$. Then this implies that the two lines
\[s\cdot  l_{r} + q\cdot m_i \ \text{and} \ q\cdot l_{r} + s\cdot m_i  \]
do not intersect $q\cdot A' + s\cdot A'$. Thus, we get $2d-1$ distinct lines 
\begin{equation} \label{skl} s\cdot l_{r} + q\cdot  l_{r}, s\cdot l_{r} + q\cdot  m_1, \dots, s\cdot l_{r} + q\cdot m_{d-1}, q\cdot l_{r} + s\cdot m_1, \dots,  q\cdot l_{r} + s\cdot m_{d-1}, \end{equation}
which do not intersect $q\cdot A' + s\cdot A'$. 
By $P(d,r-1)$, we have that 
\begin{equation}  \label{c1} |q\cdot A' + s\cdot A'| \geq (|q|+|s|+2d-2)|A'| - K(r-1), \end{equation}
where $K = K_{q,s,d}$. Moreover, by $\eqref{bs13}$, we have 
\begin{equation} \label{c2} |q\cdot p_{r} + s\cdot p_{r}| \geq (|q|+|s|)|p_{r}| - C_{q,s}. \end{equation}
Lastly, for each $1 \leq i \leq d-1$, we have the trivial bound
\[ | s\cdot p_{r} + q\cdot m_i| + |q\cdot p_{r} +s\cdot m_i| \geq 2|p_{r}|.   \]
Summing the above for all $1 \leq i \leq d-1$, we get
\begin{equation} \label{c3} \sum_{i=1}^{d-1} ( |s\cdot p_{r} + q\cdot m_{i}| + |q\cdot  p_{r} + s\cdot m_{i}| ) \geq 2(d-1)|p_r|.  \end{equation}
Combining $\eqref{c1}$, $\eqref{c2}$ and $\eqref{c3}$ with the fact that  the $2d-1$ lines mentioned in $\eqref{skl}$ do not intersect  $q\cdot A' + s\cdot A'$, we get that
\begin{align*}  |q\cdot A+s\cdot A| \ &  \geq   (|q|+|s|+2d-2)|A'|  + (|q|+|s| + 2d-2)|p_{r}| -  K(r-1)- C_{q,s} \\
& > (|q|+|s|+2d-2)|A| - Kr. \end{align*}
Hence when $A'$ is $d$-dimensional, Proposition $\ref{lines}$ holds.
\par 

Our second case is when $A'$ is $(d-1)$-dimensional. In this case, we note that as $A$ is $d$-dimensional, $l_{r}$ can not intersect the affine subspace generated by $A'$, which means that $q\cdot A' + s\cdot A'$, $s\cdot A' + q\cdot p_{r}$, $q\cdot A' +s\cdot  p_{r}$ and $q\cdot p_{r} + s\cdot p_{r}$ are pairwise disjoint sets. We claim that
\begin{equation} \label{cl}
 |s\cdot A' + q\cdot p_{r}| + |q\cdot A' +s\cdot  p_{r}| \geq 2|A'|+ 2(d-1)|p_r| - d(d+1). \end{equation}
We now prove our claim. We first assume that $|A'| \geq |p_r|$. In this subcase, we use Lemma $\ref{ruzsa}$ which implies that
\[     |s\cdot A' + q\cdot p_r| \geq |A'| + (d-1)|p_r| - d(d+1)/2,  \]
and
\[    |q\cdot  A' + s\cdot p_r| \geq |A'| + (d-1)|p_r| - d(d+1)/2.  \]
Combining these two estimates, we get $\eqref{cl}$.
\par

Thus, we now assume that $|p_r| > |A'|$. As for all $1\leq i < j \leq r-1$, the lines $s\cdot l_i + q\cdot p_{r}$, $s\cdot l_j + q\cdot p_{r}$, $q\cdot l_i +s\cdot  p_{r}$ and $q\cdot l_j + s\cdot p_{r}$ are pairwise disjoint, we have the following decomposition. 
\begin{align} \label{f1}
 |s\cdot A' + q\cdot p_{r}| + |q\cdot A' +s\cdot  p_{r}|   & = \sum_{i =1}^{r-1} |s\cdot p_i + q\cdot p_{r}| + \sum_{i =1}^{r-1} |q\cdot p_i + s\cdot p_{r}|  \nonumber  \\
& \geq  \sum_{i=1}^{r-1} (2|p_{r}| + 2|p_i| - 2)   \nonumber  \\ 
& =   2|A'| +  (2r-2)|p_{r}|  - 2(r-1) .  
 \end{align} 
 Note that as $A$ is $d$-dimensional, we must have $r \geq d$. If $d \leq r \leq d+1$, we have
\[     (2r-2) |p_r| - 2(r-1) \geq (2d-2) |p_r| - 2d.   \]
If $r > d+1$, then we observe that as $A'$ is covered by $r-1$ lines, with each line containing at least one element of $A'$, we have $|p_r| > |A'| \geq r-1$. Using this, we show that
\begin{align*} 
  (2r-2) |p_r| - 2(r-1) \ & \geq \ (2d-2)|p_r| + 2|p_r| - 2(r-1) \\
&  > \  (2d-2) |p_r| .
\end{align*}
In either case, we have
\[ (2r-2) |p_r| - 2(r-1) \geq (2d-2) |p_r| - 2d,        \] 
which, together with $\eqref{f1}$, implies that
\[     |s\cdot A' + q\cdot p_{r}| + |q\cdot A' +s\cdot  p_{r}| \geq 2|A'|+ 2(d-1)|p_r| - 2d,  \]
that is, $\eqref{cl}$ holds.
\par

Thus when $A'$ is $(d-1)$-dimensional, we have shown that $\eqref{cl}$ holds.  By $P(d-1,r-1)$, we deduce that
\[ |q\cdot A' + s\cdot A'| \geq (|q|+|s|+2d-4)|A'| - K_{q,s,d-1}(r-1).\]
From our definition of $K_{q,s,d}$, we note that $ K_{q,s,d-1} \leq  K_{q,s,d} = K$, and thus, we have
\[ |q\cdot A' + s\cdot A'| \geq (|q|+|s|+2d-4)|A'| - K(r-1).\]
Combining this with $\eqref{c2}$ and $\eqref{cl}$, we get that
\begin{align*}
 |q\cdot A + s\cdot A| & \geq (|q|+|s|+2d-2)(|A'| + |p_r|) - K(r-1) - d(d+1) - C_{q,s} \\
& \geq (|q|+|s|+2d-2)|A| - Kr.
\end{align*}
\end{proof}


\section{Proof of Theorem \ref{rtn}}

We begin this section with a preliminary lemma on sums of linear transformations of one-dimensional sets. 

\begin{lemma} \label{lin}

Let $\mathscr{L} \in \textrm{GL}_{2}( \mathbb{R})$ be a linear transformation such that $\mathscr{L}$ has no real eigenvalues. Furthermore, let $l_1$ and $l_2$ be two parallel lines in $\mathbb{R}^2$. Then for all finite subsets $A_1 \subseteq l_1$ and $A_2 \subseteq l_2$, we have
\[ |A_1 + \mathscr{L}(A_2)| = |A_1||A_2|. \]  
\end{lemma}

\begin{proof}
Let $a_1, a_3 \in A_1$  and $a_2, a_4 \in A_2$ satisfy
\[ a_1 + \mathscr{L} (a_2) = a_3 + \mathscr{L}(a_4). \]
Rearranging the above, we get that
 \[ a_1 - a_3  =  \mathscr{L}(a_4) - \mathscr{L} (a_2) = \mathscr{L} (a_4 - a_2). \]
We observe that if $a_1 - a_3$ is a non-zero vector, then $a_1 - a_3 = \lambda_1 \cdot u$ and $a_4 - a_2 = \lambda_2 \cdot u$ where $u$ is the unit vector parallel to $l_1$, and $\lambda_1$ and $\lambda_2$ are suitably chosen non-zero real numbers. Thus we have
\[\lambda_1 \cdot u = \mathscr{L} (\lambda_2 \cdot u) = \lambda_2 \cdot \mathscr{L} (u).  \] 
This implies that
\[ \mathscr{L} (u) = (\lambda_2^{-1}\lambda_1)\cdot u,  \] 
which contradicts the hypothesis that $\mathscr{L}$ has no real eigenvalues. Thus, $a_1 = a_3$, and consequently, $a_2 = a_4$. Hence, we see that all pair wise sums of the form $ a_1 + \mathscr{L} (a_2)$, with $a_1 \in A_1$ and $a_2 \in A_2$, are distinct. This implies that
\[ | A_1+ \mathscr{L}(A_2)| =  |A_1||\mathscr{L}(A_2)|  =  |A_1||A_2| . \qedhere  \] 
\end{proof}

We now prove another structure theorem which classifies sets that have a small $A + \mathscr{L}(A)$, where $\mathscr{L} \in \textrm{GL}_{2}( \mathbb{R})$ does not have real eigenvalues. 

\begin{theorem} \label{costr}
Let $\mathscr{L} \in \textrm{GL}_{2}( \mathbb{R})$ be a linear transformation such that $\mathscr{L}$ has no real eigenvalues. Furthermore, let $C > 1$ be a constant and $A$ be a finite subset of $\mathbb{R}^2$ such that
\begin{equation} \label{hyp1}
 |A + \mathscr{L}(A)|  < C|A|,
\end{equation} 
and $|A|$ is large enough. Then there is a partition $A=S\cup B$ such that the following implications hold. 

\begin{enumerate}

\item \label{itm1} There exist $r_1$ parallel lines $l_1,\dots, l_{r_1}$ such that \[ S = (l_1 \cup \dots \cup l_{r_1}) \cap A, \ \text{and} \ |B| \leq C_1 (2C)^6 {|A|}^{1-\sigma},  \]
where $C_1 > 0$ and $0 < \sigma \leq 1/2$ are constants depending only on $C$.
\item \label{itm2} We have \[ (2C)^{1/2} |A|^{1/2} > |A \cap l_1| \geq \dots \geq  |A \cap l_{r_1}| \geq  |A \cap l_1|^{1/2} \geq C_1^{-1} {|A|}^{\sigma}. \] 
\item \label{itm3} We have \[ \frac{1}{4C^{1/2}} |A|^{1/2} \leq r_1 \leq C_1 |A|^{1-\sigma}. \] 
\item \label{itm4} There exist $r_2$ parallel lines $m_1, \dots, m_{r_2}$ such that $l_1$ and $m_1$ are not parallel, and 
\[ S = (m_1 \cup \dots \cup m_{r_2}) \cap S \ \text{and} \ (8C)^{-1} r_1 \leq r_2 \leq 8C r_1 . \]
\item \label{itm5} We have \[ |A + \mathscr{L}(A)| \geq |A| \bigg( \frac{r_2}{r_1} + \frac{r_1}{r_2} + 2 \bigg)  - O(|A|^{1-\sigma}).  \] 

\end{enumerate}

\end{theorem}

We remark that Theorem $\ref{rtn}$ is a straightforward consequence of Theorem $\ref{costr}$. This can be seen by setting $C = 8$ and applying Theorem $\ref{costr}$. We combine implication $(\ref{itm5})$ from Theorem $\ref{costr}$ and the fact that 
\[ \frac{r_2}{r_1} + \frac{r_1}{r_2} \geq 2, \]
for all $r_1, r_2 > 0$, to get
\[ |A + \mathscr{L}(A)| \geq 4|A|  - O(|A|^{1-\sigma}). \] 
We set $\delta = \sigma$ to get Theorem $\ref{rtn}$. Thus it suffices to show that  Theorem $\ref{costr}$ is true.

\begin{proof}[Proof of Theorem $\ref{costr}$]

Let $|A| = n$, where $n$ is large enough and let $\mathscr{L} \in \textrm{GL}_{2}( \mathbb{R})$ be a linear transformation such that $\mathscr{L}$ has no real eigenvalues. We suppose that $|A+ \mathscr{L}(A)| \leq 2Cn $. 
\par

We now apply Theorem $\ref{str}$ with $c = 2C$. Thus we get parallel lines $l_1, l_2, \dots, l_{r_1}$ in $\mathbb{R}^2$, and constants $0 < \sigma \leq 1/2$ and $C_1 > 0$ depending only on $C$ such that 
\[ |A \cap l_1| \geq \dots \geq  |A \cap l_{r_1}| \geq  |A \cap l_1|^{1/2} \geq C_1^{-1} n^{\sigma}, \]
and 
\begin{equation} \label{con2}
|A\setminus (l_1 \cup l_2 \cup \dots \cup l_{r_1})| < C_1 (2C)^6 n^{1-\sigma}. 
\end{equation}
We set $S = A \cap  (l_1 \cup l_2 \cup \dots \cup l_{r_1})$ and $B = A \setminus S$. For ease of notation, we write $p_i = A \cap l_i$ for $1 \leq i \leq r_1$. If $|p_1| \geq (2C)^{1/2} n^{1/2}$, then by Lemma $\ref{lin}$, we have
\[ |A + \mathscr{L}(A)| \ \geq \ |p_1 + \mathscr{L}(p_1)| \ = \ |p_1|^2 \ \geq 2Cn , \]
which contradicts $\eqref{hyp1}$. Thus we must have $|p_1| < (2C)^{1/2} n^{1/2}$, and consequently, we prove $(\ref{itm1})$ and  $(\ref{itm2})$ in Theorem $\ref{costr}$.
\par

From $\eqref{con2}$, we deduce that
\[ r_1 |p_1| \ \geq \ \sum_{i=1}^{r_1} |p_i| \ = \ |S| \ = \ n - |A \setminus S| \ \geq \ n - C_1 (2C)^6 n^{1-\sigma} \ > \ n/2, \]
if $n$ is large enough. Hence
\begin{equation} \label{r1l} 
r_1 \ \geq \ \frac{1}{2} |p_1|^{-1} n \ \geq \ \frac{1}{4 C^{1/2}} n^{1/2}. 
\end{equation}
As in the proof of Theorem $\ref{main}$, there is a natural upper bound for $r_1$ in terms of $n$ as
\[ n \ \geq \ \sum_{i=1}^{r_1} |p_i| \ \geq \ r_1 C_1^{-1} n^{\sigma}. \]
Consequently, we get
\begin{equation} \label{r1u}
 r_1 \ \leq \ C_1 n^{1-\sigma}. 
\end{equation}
Thus we have proven $(\ref{itm3})$ in Theorem $\ref{costr}$. 
\par

We now divide $\mathscr{L}(S)$ into equivalence classes with respect to $l_1$, that is, we write
\begin{equation} \label{decom} 
 \mathscr{L}(S) \ = \ E_1 \cup E_2 \cup \dots \cup E_{r_2}, 
\end{equation}
where each $E_i$ lies in a unique translate of $l_1$, and $E_i \cap E_j = \emptyset$ for all $i \neq j$. As $\mathscr{L}$ does not have any real eigenvalues, $\mathscr{L}(p_i)$ is not parallel to $l_1$ for all $1 \leq i \leq r_1$. Thus each translate of $l_i$ can contain at most $r_1$ elements of $\mathscr{L}(S)$. This gives us
\[ r_1  r_2 \ \geq \ \sum_{i=1}^{r_2} |E_i| \ = \ |\mathscr{L}(S)| \ = \ |S| \ \geq \ n/2.      \]
Combining this with $\eqref{r1u}$, we deduce that
\begin{equation} \label{r2l}
r_2 \ \geq \ \frac{1}{2} r_1^{-1} n \ \geq \ \frac{1}{2}  C_1^{-1} n^{-(1-\sigma)} n \ \geq \ \frac{1}{2} C_1^{-1} n^{\sigma}.
\end{equation}
Lastly, we can trivially bound $r_2$ above by $|S|$. Our set up to apply Lemma $\ref{gs}$ is now ready. We set $A = S$, $B = \mathscr{L}(S)$, $l = l_1$ in Lemma $\ref{gs}$.  Noting that as $\mathscr{L}$ is invertible, we have $|S| = |\mathscr{L}(S)|$. Thus, $\eqref{gs1}$ implies that
\begin{equation} \label{gs2}
|S +  \mathscr{L}(S)| \ \geq \ 2|S| + |S| \bigg( \frac{r_2}{r_1} + \frac{r_1}{r_2}\bigg)  -  |S| \bigg( \frac{1}{r_1} + \frac{1}{r_2} \bigg) - (r_1 + r_2 - 1). 
\end{equation}
Using the respective lower bounds $\eqref{r1l}$ and $\eqref{r2l}$ for $r_1$ and $r_2$, we show that
\begin{equation} \label{meh12} |S| \bigg( \frac{1}{r_1} + \frac{1}{r_2} \bigg ) \ \leq \ n \bigg( \frac{8}{n^{1/2}} + \frac{2 C_1}{n^{\sigma}} \bigg) \ = \ O(n^{1-\sigma}). \end{equation}
\par

We now prove that 
\begin{equation} \label{item4p} 
 (8C)^{-1} r_1 \leq r_2 \leq 8C r_1.
\end{equation}
If the above does not hold, we see that
\[  |S| \bigg( \frac{r_2}{r_1} + \frac{r_1}{r_2}\bigg) > \frac{n}{2} (8C)  = 4Cn. \] 
We combine this with  $\eqref{gs2}$, $\eqref{meh12}$ and the fact that $r_1 + r_2 \leq 2|S|$, to get
\[ |S +  \mathscr{L}(S)| > 2|S| + 4Cn - 2|S| -   O(n^{1-\sigma}) \geq 2Cn,  \] 
when $n$ is large enough. This contradicts $\eqref{hyp1}$ and thus, $\eqref{item4p}$ must hold. 
\par

We note that $\eqref{r1u}$ and $\eqref{item4p}$ give us
\[ r_1 + r_2 \leq  (8C + 1) r_1  \leq \  (8C + 1) C_1 n^{1-\sigma}.  \]

Combining the above with $\eqref{gs2}$ and $\eqref{meh12}$, we get
\begin{align*} 
|S +  \mathscr{L}(S)| \ 
& \geq \  2|S| + |S| \bigg( \frac{r_2}{r_1} + \frac{r_1}{r_2}\bigg)  -  |S| \bigg( \frac{1}{r_1} + \frac{1}{r_2} \bigg) - (r_1 + r_2 - 1)   \\ 
& \geq \ 2|S| + |S|\bigg( \frac{r_2}{r_1} + \frac{r_1}{r_2} \bigg) - O( n^{1-\sigma} ) -  O(  n^{1-\sigma}) \\ 
& \geq \ (|A| - |B|) \bigg( \frac{r_2}{r_1} + \frac{r_1}{r_2} + 2 \bigg) - O(n^{1-\sigma}). \\
\end{align*}
Furthermore, we use $\eqref{con2}$ and $\eqref{item4p}$ to show that
\[ |B| \bigg( \frac{r_2}{r_1} + \frac{r_1}{r_2} + 2 \bigg) \leq  (16C + 2)  C_1 (2C)^6 n^{1-\sigma}. \] 
Thus, we have
\[|S +  \mathscr{L}(S)| \geq  |A| \bigg( \frac{r_2}{r_1} + \frac{r_1}{r_2} + 2 \bigg) - O(|A|^{1-\sigma}),\] 
and consequently, $(\ref{itm5})$ in Theorem $\ref{costr}$ holds. 
\par

Lastly, in $\eqref{decom}$, we decomposed $\mathscr{L}(S)$ into a disjoint union of equivalence classes $E_i$ such that each $E_i$ lies in a unique translate of $l_1$. As $\mathscr{L}$ is invertible and invertible linear transformations preserve parallel lines, we can write
\begin{align*} S 
& = \mathscr{L}^{-1}(E_1) \cup \dots \cup  \mathscr{L}^{-1}(E_{r_2}) \\
& = (m_{1} \cap S) \cup \dots \cup (m_{r_2} \cap S) \\
& =   (m_{1}  \cup \dots \cup m_{r_2} )\cap S, 
\end{align*}
where $m_{1}, \dots, m_{r_2}$ are parallel lines. Moreover, $m_1$ and $l_1$ are not parallel lines since $m_1$ is parallel to $\mathscr{L}^{-1}(l_1)$ and $\mathscr{L}$ does not have any invariant one-dimensional subspace of $\mathbb{R}^2$. This, together with $\eqref{item4p}$, proves $(\ref{itm4})$ in Theorem $\ref{costr}$.
\end{proof}

As previously mentioned, we note that Theorem $\ref{costr}$ makes partial progress towards an analogue of Conjecture $\ref{BkhSh}$ in $\mathbb{R}^2$. In particular, we set $C = 2(1+ (\det{\mathscr{L}})^{1/2})^2$ and show that if $|A + \mathscr{L}(A)| < C|A|$, then $A$ should be nicely distributed on an almost-rectangular grid formed by vectors parallel to $l_1$ and $m_1$. \\ \\
{\textbf{Funding.}}  This work was supported by a studentship sponsored by a European Research Council Advanced Grant under the European Union's Horizon 2020 research and innovation programme via grant agreement No.~695223. \\ \\
{\textbf{Acknowledgements.}} This work was done partly while the author was a visiting undergraduate at University of Bristol under the supervision of Julia Wolf and partly as a PhD student at University of Bristol under the supervision of Trevor Wooley. The author would like to thank both Julia and Trevor for their guidance and direction. The author would also like to thank the referee for many helpful comments.

\bibliographystyle{amsbracket}
\providecommand{\bysame}{\leavevmode\hbox to3em{\hrulefill}\thinspace}

\end{document}